\documentclass[12pt,a4paper,reqno]{amsart}

\usepackage[utf8]{inputenc}
\usepackage{amssymb,latexsym}
\usepackage{amsmath,amsthm}
\usepackage{enumerate}
\usepackage[all]{xy} \CompileMatrices
\usepackage{amscd}
\usepackage{slashed}
\usepackage{tikz}
\usepackage{float}
\usetikzlibrary{matrix}
\usepackage{enumitem}
\usepackage{url}
 
\usepackage[a4paper,
left=1in,right=1in,top=1.2in,bottom=1.2in,%
footskip=.25in]{geometry}

\usepackage[colorlinks=false,pdfborderstyle={/S/U/W 0}]{hyperref}

\SelectTips{cm}{12}
\theoremstyle{plain}
\newtheorem{theorem}{Theorem}[section]
\newtheorem{lemma}[theorem]{Lemma}
\newtheorem{corollary}[theorem]{Corollary}
\newtheorem{proposition}[theorem]{Proposition}

\theoremstyle{definition}
\newtheorem{definition}[theorem]{Definition}
\newtheorem{definition-theorem}[theorem]{Definition-Theorem}

\theoremstyle{remark}
\newtheorem{remark}[theorem]{Remark}

\numberwithin{equation}{section}
\setlist[itemize]{leftmargin=*}

\newcommand{\surj}{\to\kern-1.8ex\to}

\newcommand{\Conf}{\mathrm{Conf}}
\newcommand{\Sol}{\mathrm{Sol}}

\newcommand{\cR}{\mathcal{R}}
\newcommand{\frf}{\mathfrak{f}}
\newcommand{\frh}{\mathfrak{h}}

\newcommand{\cI}{\mathcal{I}}
\newcommand{\dd}{\mathrm{d}}

\setlist[itemize]{leftmargin=*}

\begin{document}

\title[Three-dimensional compact Heterotic solitons with parallel torsion]{Three-dimensional compact Heterotic solitons with parallel torsion}

\author[Andrei Moroianu]{Andrei Moroianu}
\address{Université Paris-Saclay, CNRS,  Laboratoire de mathématiques d'Orsay, 91405, Orsay, France, and Institute of Mathematics “Simion Stoilow” of the Romanian Academy, 21 Calea Grivitei, 010702 Bucharest, Romania}
\email{andrei.moroianu@math.cnrs.fr}

\author[Miguel Pino Carmona]{Miguel Pino Carmona} 
\author[C. S. Shahbazi]{C. S. Shahbazi} \address{Departamento de Matem\'aticas, Universidad UNED - Madrid, Reino de Espa\~na}
\email{mpino185@alumno.uned.es}
\email{cshahbazi@mat.uned.es}

\thanks{This research was supported by the Oberwolfach Research Fellows program at the Mathematisches Forschungsinstitut Oberwolfach (MFO). We are grateful to MFO for its hospitality and for providing a stimulating research environment. A.M. was partly supported by the PNRR-III-C9-2023-I8 grant CF 149/31.07.2023 {\em Conformal Aspects of Geometry and Dynamics}.  C.S.S. was partially supported by the research grant PID2023-152822NB-I00 of the Ministry of Science of the government of Spain.}

\begin{abstract}
We obtain a rigidity result for compact three-dimensional Heterotic solitons with parallel non-trivial torsion. We show that they are either hyperbolic three-manifolds or compact quotients of the Heisenberg group equipped with a left-invariant metric. In particular, the latter arise both as solitons with completely skew-symmetric torsion as well as with non-vanishing twistorial component.  As a corollary, we obtain the universal bound $-24$ for the scalar curvature of Heterotic solitons with parallel skew-symmetric torsion, which prevents it from being arbitrarily large. 
\end{abstract}


\maketitle
\setcounter{tocdepth}{1} 

 
\section{Introduction}
\label{sec:intro}


The Heterotic soliton system is a system of differential equations, depending on a positive real parameter $\kappa$, that originates in Heterotic string theory \cite{BRI, BRII, Papadopoulos:2024tgs} and was proposed in \cite{Moroianu:2021kit,Moroianu:2023jof} as a natural gauge-theoretic differential system in Riemannian geometry with torsion. The Heterotic soliton system, which we present in detail in Definition \ref{def:motionsugratorsion} below, couples a Riemannian metric $g$ to a closed 1-form $\varphi$, a 3-form $H$, and a metric connection $D$ with (not necessarily skew-symmetric) torsion. The connection $D$ appears in the system through a natural quadratic term in its curvature tensor, which is one of the main novelties of the system. The Heterotic soliton system includes, as notable particular cases, the celebrated Hull-Strominger system \cite{GarciaFernandez,Hull,Picard,Strominger} and the Heterotic $\mathrm{G}_2$ \cite{Marisa,Friedrich:2001nh} and $\mathrm{Spin}(7)$ \cite{Marisa,Ivanov} systems with trivial gauge bundle. In fact, the Heterotic soliton system can be understood as the system of \emph{integrability conditions} associated with these spinorial differential systems \cite{Gillard:2004xq}. However, solutions to the Heterotic soliton system can be remarkably more general, as proven in \cite{Moroianu:2023jof}, where the first examples of non-supersymmetric, not locally supersymmetric, Heterotic compactification backgrounds were constructed in four dimensions. Alternatively, the Heterotic soliton system can be considered as a \emph{corrected} generalized Ricci soliton system \cite{FernandezStreetsLibro,Oliynyk} via the inclusion of the quadratic curvature term prescribed by Heterotic string theory. As illustrated and explained in \cite{StreetsSoliton,StreetsUstinovskiySoliton,StreetsUstinovskiySolitonII}, the generalized Ricci soliton system promises to have important applications in the geometrization and classification of compact complex surfaces, which would be interesting to understand in the context of the Heterotic soliton system. The primary objective of this article is to classify the three-dimensional compact solutions of the Heterotic soliton system with non-trivial, parallel, but not necessarily skew-symmetric, torsion. Our classification results can be summarized in terms of the following \emph{rigidity} result. If $(g,\varphi,H,D)$ is a Heterotic soliton with non-trivial parallel torsion on a compact three-manifold $M$ (cf. Definition \ref{def:motionsugratorsion} below), then $(M,g)$ is a Riemannian manifold of the following type: 
\begin{itemize}
\item A compact quotient of the three-dimensional Heisenberg group equipped with a left-invariant metric.

\item An Einstein manifold of strictly negative curvature (i.e. up to homothety, a compact quotient of the hyperbolic space $\mathbb{H}^3$).
\end{itemize}

\noindent
The reader is referred to Theorems \ref{thm:Sasaki} and \ref{thm:skewtorsion} for more details. The proof involves Bochner-type identities together with repeated use of the maximum-minimum principle and a detailed analysis of the higher curvature terms involved in the Heterotic soliton system, which require extensive manipulations. As a corollary to this classification, we obtain the following normalized \emph{universal} bound for three-dimensional compact Heterotic solitons of hyperbolic type.
\begin{corollary}
\label{cor:boundcomplete}
Let $(g,\varphi,H,D)$ be a Heterotic soliton with non-trivial parallel torsion on a compact three-manifold $M$. If $(M,g)$ is hyperbolic then:
\begin{eqnarray*}
 -24 < \kappa s_g < 0\, . 
\end{eqnarray*}
\end{corollary}

\noindent
Here $\kappa$ is a positive real constant (typically called the \emph{Regge slope} in the physics literature \cite{Polchinski:1998rq}), which from our point of view becomes a parameter of the system. Hence, the scalar curvature of a hyperbolic Heterotic soliton with parallel torsion is bounded in $\kappa^{-1}$ \emph{units}. Interestingly enough, this bound does not depend on the underlying compact hyperbolic three-manifold $M$. It would be interesting to understand if the allowed interval $(-24,0)$ for $\kappa s_g$ is preserved in the case of hyperbolic Heterotic solitons of not necessarily parallel torsion, and more generally when taking into account further higher order corrections consistently \cite{Melnikov:2014ywa}.\medskip 

\noindent
Note that compact quotients of the Heisenberg group occur as solutions of the Heterotic soliton system with both a connection $D$ with fully skew-symmetric torsion and a connection $D$ with non-zero twistorial component. In the former case, such a Heterotic soliton has a scalar curvature satisfying $2\kappa s_g = -1$, while in the latter, it can be chosen arbitrarily. Hence, the previous corollary can be refined as follows.

\begin{corollary}
\label{cor:boundcompleteII}
Let $(g,\varphi,H,D)$ be a Heterotic soliton with connection $D$ having non-trivial skew-symmetric torsion. Then:
\begin{eqnarray*}
 -24 < \kappa s_g < 0\, . 
\end{eqnarray*}
\end{corollary}

\noindent
The bound $-24$ can be saturated by a Heterotic soliton of vanishing torsion on an Einstein manifold of scalar curvature $\kappa s_g = -24$. Understanding the physical reasons for the occurrence of the bound $-24$ is beyond the scope of this paper and is left for experts in the field.  \medskip

\noindent
Note that in this article, we only consider Heterotic solitons with parallel \emph{non-trivial} torsion. The case of \emph{vanishing} torsion is, surprisingly enough, remarkably more complicated and will be considered in a forthcoming publication.


\section{Preliminaries}
\label{sec:formulae}


Let $M$ be an oriented $d$-dimensional manifold $M$ equipped with a Riemannian metric $g$. We denote by $\langle \cdot , \cdot\rangle_g$ the determinant inner product induced by $g$ on the exterior algebra bundle of $M$ (which differs by a factor $\tfrac1{k!}$ from the tensor scalar product on $\wedge^kM$), and by $\vert \cdot \vert_g$ its associated norm. For every metric connection $D$ on $(M,g)$, we will denote by $\cR^{D}$ its curvature tensor, which in our conventions is defined by the following formula:
\begin{equation}\label{rd}
\cR^{D}_{X,Y} Z = D_{X} D_{Y} Z - D_{Y} D_{X} Z - D_{[X,Y]} Z
\end{equation}

\noindent
for every vector fields $X , Y, Z\in \mathfrak{X}(M)$. We understand the curvature tensor $\cR^{D}$ of $D$ as a section of $\wedge^2 M\otimes \wedge^2 M$, upon identification of skew-symmetric endomorphisms with 2-forms, using the Riemannian metric. In our conventions the norm of $\cR^{D}$ is explicitly given by:
\begin{eqnarray*}
\vert \cR^{D} \vert_g^2 : = \langle \cR^{D} , \cR^{D} \rangle_g = \frac{1}{2} \sum_{i,j=1}^d \langle \cR^{D}_{e_i,e_j} , \cR^{D}_{e_i,e_j} \rangle_g = \frac{1}{4}\sum_{i,j,k,l=1}^d \cR^{D}_{e_i,e_j}(e_k,e_l)\, \cR^{D}_{e_i,e_j}(e_k,e_l)
\end{eqnarray*}

\noindent
in terms of any local orthonormal frame $e_i$, where $\cR^{D}_{e_i,e_j} = e_j \lrcorner e_i\lrcorner \cR^{D}$ denotes evaluation of $e_i$ and $e_j$ in the \emph{first factor} $\wedge^2 M$ of $\cR^{D}$ and $\cR^{D}(e_i,e_j)$ denotes evaluation of $e_i$ and $e_j$ on the \emph{second factor} $\wedge^2 M$ of $\cR^{D}$. When $D = \nabla^g$ is the Levi-Civita connection of $(M,g)$, we set $\cR^g := \cR^D$. We define a symmetric bilinear form naturally associated to the curvature tensor of $D$ by:
\begin{equation}\label{rdrd}
(\cR^{D}\circ_g \cR^{D})(X,Y) := \langle X\lrcorner \cR^{D}, Y\lrcorner \cR^{D}\rangle_g = \frac{1}{2}\sum_{i,j,k=1}^d \cR^{D}_{X , e_i}(e_j,e_k) \cR^{D}_{Y e_i}(e_j, e_k)\, , 
\end{equation} 

\noindent
for every $X, Y \in \mathfrak{X}(M)$. This defines, associated to every metric connection $D$, a symmetric two-form $\cR^{D}\circ_g \cR^{D}\in \Gamma(T^{\ast}M\odot T^{\ast}M)$, quadratic in the curvature of $D$ that defines one of the most important terms appearing in the Heterotic soliton system.   Furthermore, associated with the curvature tensor of $D$, we define the 4-form:
\begin{equation*}
\langle \cR^{D}\wedge \cR^{D}\rangle_g = \frac{1}{2} \sum_{i,j=1}^d \cR^{D}(e_i,e_j)\wedge \cR^{D}(e_i,e_j)\, ,
\end{equation*}

\noindent
obtained by taking the wedge product on the \emph{first factor}  and the norm of the \emph{second factor} of $\cR^{D}$ in $\wedge^2 M\otimes \wedge^2 M$. Given such a connection $D$ on $TM$, for every $k,l\in \mathbb{N}$ we denote by:
\begin{equation*}
\dd_D \colon \Omega^k(M)\otimes\Omega^l(M)\to \Omega^{k+1}(M)\otimes\Omega^l(M)    
\end{equation*}

\noindent
the \emph{covariant} exterior derivative associated to $D$. Its formal adjoint is denoted by:
\begin{eqnarray*}
\dd_D^{\ast} \colon \Omega^{k+1}(M)\otimes\Omega^l(M)\to \Omega^{k}(M)\otimes\Omega^l(M)    
\end{eqnarray*}

\noindent
and plays a fundamental role in the definition of the Heterotic soliton system. For every 3-form $H\in \Omega^3(M)$, we define the symmetric bilinear form:
\begin{eqnarray*}
(H\circ_g H)(X,Y) :=   \langle X \lrcorner H , Y\lrcorner H\rangle_g = \frac{1}{2} \sum_{i,j=1}^d H(X,e_i,e_j) H(Y,e_i,e_j)\, , \quad \forall \ X, Y \in TM\, .
\end{eqnarray*}

\noindent
Similarly, we define:
\begin{equation*}
(\mathrm{Ric}^g\circ_g\mathrm{Ric}^g)(X,Y) := \langle \mathrm{Ric}^g(X) , \mathrm{Ric}^g(Y) \rangle_g\, , \qquad \forall \ X, Y \in TM\, ,
\end{equation*}

\noindent
where $\mathrm{Ric}^g$ denotes the Ricci tensor of $g$. If $M$ is three-dimensional, then the Riemannian tensor $\cR^g$ of $g$ can be written in terms of its Ricci tensor $\mathrm{Ric}^g$ as follows:
\begin{equation}
\label{eq:appR3d}
\cR^g_{X,Y} = \tfrac12{s_g} X \wedge Y + Y\wedge \mathrm{Ric}^g(X) + \mathrm{Ric}^g(Y)\wedge X\, , \qquad \forall \ X, Y \in TM\, ,
\end{equation}

\noindent
where $s_g$ is the scalar curvature of $g$. In particular, using the previous formula, it can be seen that the contraction $\cR^g\circ_g \cR^g$, is given by:
\begin{equation}
\label{eq:appRR3d}
\cR^g\circ_g \cR^g =  -  \mathrm{Ric}^g\circ \mathrm{Ric}^g +  s_g \mathrm{Ric}^g + (\vert \mathrm{Ric}^g\vert^2_g - \tfrac12{s_g^2}) g\, ,
\end{equation}

\noindent
where we have defined:
\begin{equation*}
\mathrm{Ric}^g\circ_g \mathrm{Ric}^g(X,Y) := g(\mathrm{Ric}^g(X),\mathrm{Ric}^g(Y))\, , \qquad  \forall\ X, Y\in TM\, .
\end{equation*}

\noindent
In particular, the norm of $\cR^g$ is given by $\vert \cR^g \vert_g^2 = \tfrac{1}{2} \mathrm{Tr}_g(\cR^g\circ_g \cR^g)  =  \vert \mathrm{Ric}^g \vert_g^2 - \tfrac{1}{4} s_g^2$. Given a three-form $H\in \Omega^3(M)$ on $M$, we denote by:
\begin{equation*}
\nabla^{g,H} = \nabla^g + \tfrac{1}{2}H\, ,
\end{equation*}

\noindent
the unique metric connection with completely skew-symmetric torsion $H$. The Ricci tensor of such a connection, which appears explicitly in the Heterotic soliton system, will be denoted by $\mathrm{Ric}^{g,H}\in \Gamma(T^{\ast}M\otimes T^{\ast}M)$. A computation gives the following formula for $\mathrm{Ric}^{g,H}$:
\begin{eqnarray}
\label{eq:Ricgh}
\mathrm{Ric}^{g,H} = \mathrm{Ric}^g - \tfrac{1}{2} H\circ_g H + \tfrac{1}{2} \delta^g H\, ,
\end{eqnarray}

\noindent
where $\delta^g$ is the formal adjoint of the exterior derivative.


\section{The Heterotic soliton system}
\label{sec:HS}


We proceed to introduce the Heterotic soliton system. Our notation and conventions are explained in the previous section. 
\begin{definition}
\label{def:motionsugratorsion}
Let $\kappa > 0$ be a non-negative real constant. The \emph{Heterotic soliton system} on a manifold $M$ is the following system of partial differential equations:
\begin{eqnarray}
& \mathrm{Ric}^{g,H} +  \nabla^{g,H}\varphi  + \kappa\, \cR^{D} \circ_g\cR^{D} = 0 \, , \qquad  \dd^{\ast}_D\cR^{D} + \varphi\lrcorner\cR^{D} = 0 \nonumber \\ 
& \delta^g\varphi + |\varphi|^2_g  - |H|^2_g  +  \kappa\, |\cR^{D}|^2_{g} = 0 \label{eq:motionsugratorsion}	
\end{eqnarray}
	
\noindent
together with the \emph{Bianchi identity}:
\begin{equation}
\label{eq:BianchiIdentity}
\dd H + \kappa \langle \cR^{D}\wedge \cR^{D}\rangle_g = 0
\end{equation}
	
\noindent
for tuples $(g,\varphi,H,D)$, where $g$ is a Riemannian metric on $M$, $\varphi\in \Omega^1(M)$ is a closed 1-form, $H\in \Omega^3(M)$ is a 3-form, and $D$ is a connection on $TM$ compatible with $g$. A \emph{Heterotic soliton} is a tuple $(g,\varphi, H,D)$ satisfying the Heterotic soliton system.
\end{definition}

\noindent
In the previous definition, and throughout the article, we are identifying, for every given tuple $(g,\varphi,H,D)$, vectors and 1-forms via the Riemannian metric $g$. Furthermore, as explained in Section \ref{sec:formulae}, $\varphi\lrcorner \cR^D$ denotes the contraction with $\varphi$ in the first $\wedge^2 M$ component of $\cR^D\in\wedge^2 M\otimes\wedge^2 M$.
\begin{remark}
The Bianchi identity \eqref{eq:BianchiIdentity} is not an \emph{identity} but an equation that needs to be solved. The terminology comes from the physics community and is nowadays standard also in the mathematics community.
\end{remark}

\noindent
Combining the first and third equations in \eqref{def:motionsugratorsion}, we obtain the following  equation, which does not depend on $\kappa$ and is useful for applications:
\begin{equation}
\label{eq:dilatontrace}
s_g = 3 \delta^g \varphi + 2\vert\varphi\vert^2_g - \tfrac{1}{2} \vert H\vert^2_g\, . 
\end{equation}

\noindent
The Heterotic soliton system corresponds to the equations of motion of bosonic Heterotic supergravity with a trivial gauge bundle at first order in the string slope parameter $\kappa$ \cite{BRI, BRII}. The symmetric part of the first equation in \eqref{eq:motionsugratorsion} is the \emph{Einstein equation} of the system, whereas its skew-symmetric projection is typically referred to as the \emph{Maxwell equation}. On the other hand, the second equation in \eqref{eq:motionsugratorsion} is the \emph{Yang-Mills equation} for the auxiliary connection $D$, whereas the third equation in \eqref{eq:motionsugratorsion} is commonly called the \emph{dilaton equation}. We will denote the configuration space and solution space of the Heterotic soliton system by $\Conf(M)$ and $\Sol_{\kappa}(M)$, respectively. We will refer to the variable $D$ as the \emph{auxiliary connection} of the given Heterotic soliton. 

\begin{remark}
\label{remark:H0}
Note that if $M$ is compact and $H=0$, then integrating the dilaton equation in \eqref{eq:motionsugratorsion} we obtain that $\varphi = 0$, $D$ is flat and $g$ is Ricci flat. We define such Heterotic soliton as being \emph{trivial}. 
\end{remark}

\noindent
Equations \eqref{eq:motionsugratorsion}	 and \eqref{eq:BianchiIdentity}	define the Heterotic soliton system in any dimension. In the following, we will assume that the dimension of $M$ is 3 and $M$ is oriented. This dimensional assumption simplifies the system notably, as the Bianchi identity \eqref{eq:BianchiIdentity} becomes trivially satisfied, the Riemann tensor is completely determined by the Ricci tensor, and other natural tensorial identifications hold. In particular, we will denote the elements of the configuration space $(g,\varphi, H,D) \in \Conf(M)$ simply by $(g,\varphi,\frh, D)\in \Conf(M)$, where $\mathfrak{h} := \ast_g H\in C^{\infty}(M)$. This allows us to easily integrate the Maxwell equation of the system, namely the skew-symmetric part of the Einstein equation, the first equation in \eqref{eq:motionsugratorsion}.

\begin{lemma}
\label{lemma:Maxwellintegrated}
Let $(g,\varphi,\frh, D)$ be a non-trivial Heterotic soliton on a compact connected three-manifold $M$. Then, there exists a function $\phi \in C^{\infty}(M)$ such that $\varphi = \dd \phi$ and $\frh = c\, e^{\phi}$ for a non-zero constant $c\in \mathbb{R}^{\ast}$. In particular, $\frh$ is constant if and only if $\varphi = 0$ identically on $M$.
\end{lemma}
 
\begin{proof}
Let $(g,\varphi,\frh, D)$ be a non-trivial three-dimensional Heterotic soliton. Denoting by $\nu_g$ the Riemannian volume form associated to $g$, the skew-symmetric part of the first equation in \eqref{eq:motionsugratorsion} is equivalent to:   
\begin{equation*}
0 = \delta^{g} H + H(\varphi) = - \ast_g \dd \ast_g (\frh\, \nu_g) + \frh\, \nu_g (\varphi) = - \ast_g \dd \frh + \frh \ast_g \varphi\, ,
\end{equation*}

\noindent
where we have used Equation \eqref{eq:Ricgh} and $\delta^g$ denotes the codifferential with respect to $g$. Hence, the Maxwell equation of the three-dimensional Heterotic system reduces to:
\begin{equation}
\label{eq:ReducedMaxwell}
\dd \frh = \frh \,\varphi\, . 
\end{equation}

\noindent
By Remark \ref{remark:H0}, setting $\frh = 0$ contradicts the non-triviality of $(g,\varphi,\frh,D)$. Instead, we prove that $\frh$ is necessarily nowhere vanishing. To do this, assume that $\frh(m) = 0$ for a given $m\in M$. For every smooth curve:
\begin{equation*}
\gamma \colon \cI \to M \, ,\quad t\mapsto \gamma_t
\end{equation*}

\noindent
where $\cI$ is an interval containing $0$ and $\gamma(0) = m$, the first equation in \eqref{eq:ReducedMaxwell} implies:
\begin{equation*}
\partial_t(\frh\circ \gamma_t) = (\gamma^{\ast}\varphi)(\partial_t  (\frh\circ \gamma_t))
\end{equation*}

\noindent
which is a linear ordinary differential equation for $\frh\circ \gamma\colon \cI\to \mathbb{R}$. Since $\frh(m)=0$, the existence and uniqueness of solutions to the previous ordinary differential equation imply that $\frh\circ \gamma_t = 0$ for all $t\in \cI$. Since this holds for every such $\gamma\colon \cI\to M$ and $M$ is connected, we conclude that $\frh$ vanishes identically, a contradiction. Hence $\frh$ is nowhere vanishing, and Equation \ref{eq:ReducedMaxwell} reduces to $\varphi=\dd \phi$ for a function $\phi \in C^{\infty}(M)$. Consequently $\frh = c e^{\phi}$ for a non-zero constant $c\in\mathbb{R}^{\ast}$.  
\end{proof}

\noindent
By the previous lemma, in three dimensions, the Maxwell equation is completely solved, and consequently, the three-dimensional Heterotic soliton system reduces to:
\begin{eqnarray}
\label{eq:HeteroticSystem3dI}	
& \mathrm{Ric}^{g} +  \nabla^{g}\varphi - \frac{\frh^2}{2} g + \kappa\, \cR^{D} \circ_g\cR^{D} = 0 \, , \qquad \dd^{\ast}_D\cR^{D} + \varphi\lrcorner\cR^{D} = 0\, , \\
& \delta^g\varphi + |\varphi|^2_g  - \frh^2  +  \kappa\, |\cR^{D}|^2_{g} = 0\, ,\qquad \dd\frh=\frh\varphi\, , \label{eq:HeteroticSystem3dII}	
\end{eqnarray}

\noindent
for tuples $(g,\varphi,\frh,D)\in \Conf(M)$. Given $(g,\varphi,\frh,D)\in \Conf(M)$, since $D$ is compatible with $g$, we can always write:
\begin{eqnarray*}
D =\nabla^{g, \mathbb{A}} := \nabla^g + \mathbb{A}\, ,    
\end{eqnarray*}

\noindent
for a unique \emph{contorsion tensor} $\mathbb{A} \in \Gamma(\wedge^1 M \otimes \wedge^2 M)$. We say that a contorsion tensor $\mathbb{A}\in \Gamma(\wedge^1 M \otimes \wedge^2 M)$ on $(M,g)$ is parallel if:
\begin{equation}
\label{eq:parallelicityA}
\nabla^{g,\mathbb{A}} \mathbb{A} = 0\, ,
\end{equation}

\noindent
and we will refer to such Heterotic soliton as having \emph{parallel torsion}. In the following, we will exclusively consider the system of equations \eqref{eq:HeteroticSystem3dI} and \eqref{eq:HeteroticSystem3dII} for tuples $(g,\varphi,\frh, D)$ for which the contorsion tensor of $D$ is parallel with respect to $D$. Given a non-trivial Heterotic soliton $(g,\varphi,\frh, D)$, a direct inspection of the Heterotic soliton system reveals that $(g,\varphi,-\frh, D)$ is also a non-trivial Heterotic soliton. On the other hand, by Lemma \ref{lemma:Maxwellintegrated} we have that $\frh$ is nowhere vanishing. Since $M$ is assumed to be connected, we can assume that $\frh$ is nowhere vanishing and positive, in which case we have $\varphi = \dd\log(\frh)$. Consequently, we will refer to three-dimensional Heterotic solitons simply as triples of the form $(g,\frh, D)$. Furthermore, since $D$ is determined by $g$ and its contorsion tensor $\mathbb{A}$, we will occasionally refer to $(g,\frh, D)$ simply by $(g,\frh, \mathbb{A})$.


\section{Parallel torsion on compact Riemannian three-manifolds}
\label{sec:contorsion3d}


As a preliminary step to our study of the three-dimensional Heterotic soliton system, in this section we consider some basic properties of compact Riemannian three-manifolds $(M,g)$ equipped with a parallel contorsion tensor. Given a contorsion tensor $\mathbb{A} \in \Gamma(\wedge^1 M \otimes \wedge^2 M)$ on $(M,g)$, there exists a unique tensor $A \in \Gamma(T^{\ast} M \otimes T^{\ast}M)$ such that:
\begin{equation*}
\mathbb{A}_X = \ast_g (A(X))\, , \qquad \forall \ X\in TM\, . 
\end{equation*}

\noindent
We decompose $A$ into its orthogonal irreducible components:
\begin{equation*}
A = \tfrac13{\mathrm{Tr}_g(A)}  g + \Theta + \ast_g \zeta\, ,     
\end{equation*}
where $\Theta$ is a traceless symmetric tensor and $\zeta \in \Omega^1(M)$ is a 1-form. From the general formulas valid in dimension 3:
\begin{equation*}
(\ast_g \zeta)(X)=\ast_g (\zeta\wedge X),\qquad \ast_g^2=\mathrm{Id}_{\wedge^*(M)}\, ,    
\end{equation*}

\noindent
we thus have:
\begin{equation*}
\mathbb{A}_X = \tfrac13{\mathrm{Tr}_g(A)} \ast_g X + \ast_g\Theta(X) +  \zeta\wedge X\, , \qquad \forall \ X \in TM\, ,
\end{equation*}
which in turn implies for every $X,Y\in\mathfrak{X}(M)$:
\begin{equation*}
\nabla^{g , \mathbb{A}}_X Y = \nabla^g_X Y + \tfrac13{\mathrm{Tr}_g(A)} \ast_g (X \wedge Y) + \ast_g(\Theta(X) \wedge Y) +  \zeta(Y)  X - g(X,Y) \zeta\, . 
\end{equation*} 

\noindent
Since there is a one-to-one correspondence between $A$ and $\mathbb{A}$ on $(M,g)$, in the following we will denote $\nabla^{g,\mathbb{A}}$ simply by $\nabla^{g,A}$. Similarly, we will refer to three-dimensional Heterotic solitons $(g,\frh, \mathbb{A})$ as triples of the form $(g,\frh, A)$, where $A \in \Gamma(T^{\ast}M \otimes T^{\ast} M)$. By the previous discussion, the parallelism condition \eqref{eq:parallelicityA} for $A$ is equivalent to its irreducible components being parallel:
\begin{equation}
\label{eq:Afinaldecomposition}
\nabla^{g, A}\zeta = 0\, , \qquad \dd\mathrm{Tr}_g(A) = 0\, , \qquad \nabla^{g,A} \Theta = 0\, . 
\end{equation}

\noindent
This observation leads us to the following mutually exclusive possibilities.

\begin{lemma}
\label{lemma:cases}
Let $(M,g)$ be a compact Riemannian three-manifold equipped with a parallel contorsion tensor $A\in \Gamma(\wedge^1 M \otimes \wedge^2 M)$. Then, one and only one of the following possibilities occurs:
\begin{enumerate}[leftmargin=*]
    \item $\nabla^{g,A}$ is locally irreducible, in which case the contorsion tensor $A$ is a constant multiple of the metric. Equivalently, $\mathbb{A}$ is a constant multiple of the volume form of $(M,g)$.
    \item $\nabla^{g,A}$ is locally reducible and non-flat, in which case, up to replacing $M$ by a double cover, there exists a unit norm 1-form $\xi\in \Omega^1(M)$ and constants $\alpha,\beta, \gamma \in \mathbb{R}$ such that for every $ X , Y \in \mathfrak{X}(M)$:
    \begin{equation}
    \label{eq:nablareducible}
    \nabla^{g,A}_X Y = \nabla^g_X Y + \alpha \ast_g (X\wedge Y) + \beta\, Y\lrcorner (\xi\wedge X) + \gamma\, \xi(X) \ast_g (\xi\wedge Y)\, . 
    \end{equation}
    \item $\nabla^{g,A}$ is flat.
\end{enumerate}
\end{lemma}

\begin{proof}
(1) Suppose first that $\nabla^{g, A}$ is locally irreducible. Then, there are no 1-forms parallel with respect to $\nabla^{g,A}$, whence Equation \eqref{eq:Afinaldecomposition} immediately implies $\zeta = 0$. On the other hand, since $\Theta$ is parallel, symmetric, and traceless, it follows that it either vanishes identically, or has at least one simple eigenvalue, say $c_0$, which is constant. In the latter case, the eigenvectors of this simple eigenvalue $c_0$ form a real line bundle preserved by $\nabla^{g, A}$,
contradicting the fact that $\nabla^{g,A}$ is locally irreducible. Thus $\Theta = 0$ on $M$. \medskip

\noindent
(2) Suppose now that $\nabla^{g, A}$ is locally reducible and not flat. Then, up to passing to a double cover if necessary, $M$ admits a unit norm 1-form $\xi$ satisfying  $\nabla^{g ,A}\xi = 0$. This 1-form is unique modulo sign, since otherwise $\nabla^{g ,A}$ would be flat, a contradiction. Hence, pulling back to this double covering we have $\zeta = \beta\, \xi $ for a real constant $\beta \in \mathbb{R}$, as well as:
\begin{equation*}
\Theta = \gamma\, (\xi \otimes \xi - \tfrac{1}{3} g) 
\end{equation*}

\noindent
for a real constant $\gamma \in \mathbb{R}$. Therefore, we can write:
\begin{equation*}
A = \alpha^{\prime} g + \beta \ast_g \xi + \gamma (\xi \otimes \xi - \tfrac{1}{3} g) = \alpha\, g + \beta \ast_g \xi + \gamma\, \xi\otimes \xi\, ,
\end{equation*}

\noindent
where we have set $\alpha :=  \alpha^{\prime} -\tfrac13\gamma$.\medskip

\noindent
(3) Finally, if neither (1) nor (2) hold, then $\nabla^{g,A}$ is flat.
\end{proof}

\begin{remark}
In the following, and for simplicity in the exposition, we will implicitly work on $M$ or any of its double covers at our convenience.
\end{remark}

\noindent
Note that case $(2)$ in the previous lemma reduces the parallelism condition of the torsion simply to $\nabla^{g, A}\xi = 0$. On the other hand, regarding the classification of non-trivial Heterotic solitons on a compact manifold, case $(3)$ in the previous lemma cannot occur for their auxiliary connection, as shown in the next result:

\begin{proposition}
A Heterotic soliton with flat auxiliary connection on a compact and oriented three-manifold is trivial.  
\end{proposition}

\begin{proof}
If $(g,\varphi,\frh,D)$ has flat auxiliary connection $D$, then the dilaton equation, namely the third equation in \eqref{eq:motionsugratorsion}, reduces to:
\begin{equation*}
\delta^g\varphi + |\varphi|^2_g  =  \frh^2\, . 
\end{equation*}

\noindent
Assume that $(g,\varphi,\frh,D)$ is non-trivial. Then, by Lemma \ref{lemma:Maxwellintegrated} we have $\varphi = \dd\phi$ and $\frh = c e^{\phi}$ for a non-zero constant $c$. Plugging these expressions into the previous equation gives:
\begin{equation*}
 \Delta_g e^{-\phi} + c^2 e^{\phi} = 0\, . 
\end{equation*}

\noindent
Integrating over $M$ gives $c=0$, a contradiction.
\end{proof}

\begin{remark}
\label{remark:FinalCases}
By the previous proposition, to classify Heterotic solitons with parallel torsion, we only need to consider the first two cases in Lemma \ref{lemma:cases} for their auxiliary connection. It is convenient to regroup them and split them as follows:
\begin{enumerate}[leftmargin=*]
    \item[(A)] Heterotic solitons with non-flat and reducible auxiliary connection $D$ whose contorsion tensor is parallel and not purely skew symmetric. This is case $(2)$ of Lemma \ref{lemma:cases} supplemented with the condition that $\beta^2 + \gamma^2 \neq 0$. 

    \item[(B)] Heterotic solitons with purely skew-symmetric parallel torsion, be it irreducible or not. This covers case $(1)$ in Lemma \ref{lemma:cases}, together with the subcase $\beta=\gamma= 0$ of $(2)$ in the same lemma. 
\end{enumerate}
\end{remark}

\noindent
In the ensuing sections, we will consider these two cases separately.


\section{Generic reducible parallel torsion}
\label{sec:ReducibleA}


In this section, we consider three-dimensional Heterotic solitons whose auxiliary connection $D$ has a non-flat but reducible parallel contorsion of not purely skew-symmetric type, corresponding to case (A) in Remark \ref{remark:FinalCases}. We will refer to such a connection as having \emph{generic reducible parallel torsion}. We begin with some preliminary results on the general structure of compact three-manifolds equipped with such a metric connection with generic reducible, parallel torsion.


\subsection{Riemannian compact three-manifolds with generic reducible parallel torsion}
\label{sec:Reduciblepreliminaries}


Let $(M,g)$ be a compact Riemannian three-manifold equipped with a metric connection $\nabla^{g, A}$ with parallel contorsion as given in Equation \eqref{eq:nablareducible} for constants $\alpha , \beta, \gamma\in \mathbb{R}$ satisfying $\beta^2 + \gamma^2 \neq 0$. Setting $Y=\xi$ in \eqref{eq:nablareducible}, we obtain:
\begin{equation*}
\nabla^{g,A}_X \xi = \nabla^g_X\xi + \alpha \ast_g (X\wedge \xi) + \beta\, \xi \lrcorner (\xi\wedge X) = 0\, ,
\end{equation*}
or, equivalently:
\begin{eqnarray}
\label{eq:covariantxi}
\nabla^g\xi = \alpha \ast_g \xi + \beta \, (\xi \otimes \xi - g)\, ,
\end{eqnarray}

\noindent
which is thus equivalent to $\nabla^{g,A}$ having parallel contorsion. Taking the skew-symmetric and trace projections of this equation, we obtain:
\begin{equation*}
\dd \xi = 2 \alpha \ast_g \xi \, , \qquad \delta^g\xi = 2\beta\, . 
\end{equation*}
Since $M$ is compact this gives $\beta= 0$, and consequently the connection $\nabla^{g,A}$ simplifies to:
\begin{equation}
\label{eq:finalnablaA}
\nabla^{g,A}_X Y = \nabla^g_X Y + \alpha \ast_g (X\wedge Y) + \gamma\, \xi(X) \ast_g (\xi\wedge Y)  \, ,
\end{equation}

\noindent
for all vector fields $X , Y \in \mathfrak{X}(M)$. 
Therefore, the contorsion of $\nabla^{g,A}$ is given by:
\begin{equation}
\label{eq:finalA} 
A = \alpha\, g + \gamma\, \xi\otimes \xi\, . 
\end{equation}
Equation \eqref{eq:covariantxi} now reads 
\begin{equation}
\label{eq:finalnablaxi} 
\nabla^{g} \xi =  \alpha \ast_g \xi\, .
\end{equation}

\noindent
Since $\beta=0$ and we are assuming $\beta^2 + \gamma^2 \neq 0$ to guarantee that the connection $D$ does not have purely skew-symmetric torsion, it follows that $\gamma\neq 0$, a condition that we will assume for the remainder of this section. We will now show that \eqref{eq:finalnablaxi} allows one to compute the full curvature tensor of $M$ in terms of the scalar curvature $s_g$.

\begin{lemma}
\label{lemma:Riccixi}
 Let $(M,g)$ be a compact Riemannian three-manifold carrying a unit length vector field $\xi$ satisfying Equation \eqref{eq:finalnablaxi}. Then:
 \begin{equation}\label{rg}
\cR^{g}_{X,Y} =- \alpha^2 (X\wedge Y) + (3\alpha^2-\tfrac12 s_g)\, \langle \ast_g\xi,X\wedge Y\rangle_g  \ast_g\xi\,,\qquad\forall\ X,Y\in TM\, .
\end{equation}
In particular, the Ricci tensor of $g$ is given by: 
 \begin{equation}
 \label{ricg}
 \mathrm{Ric}^g = (\tfrac12s_g - \alpha^2) g + (3 \alpha^2 - \tfrac12s_g) \xi\otimes \xi  \, . 
 \end{equation}
\end{lemma}

\begin{proof}
Using \eqref{eq:finalnablaxi} we compute for every vector fields $X,Y$ parallel at a point:
\begin{eqnarray*}
\cR^{g}_{X,Y}\xi &=& \alpha(\nabla_X(Y\lrcorner(\ast_g\xi) - \nabla_Y(X\lrcorner(\ast_g\xi))  \\
&=&\alpha^2(Y\lrcorner(\xi\wedge X)-Y\lrcorner(\xi\wedge X)) = -\alpha^2\xi\lrcorner(X\wedge Y)\, .   
\end{eqnarray*}

\noindent
Using the pair symmetry of $\cR^g$, this gives: 
\begin{equation}
\label{rxi}
\cR^g_{\xi,X}=-\alpha^2\xi\wedge X,\qquad\forall\ X\in TM\, .
\end{equation}

\noindent
Let $e_1,e_2$ be local vector fields on $M$ such that $\xi,e_1,e_2$ is a local orthonormal frame of $TM$. From \eqref{rxi} we obtain that the 2-form $\cR^g_{e_1,e_2}$ is orthogonal to $\xi\wedge e_1$ and $\xi\wedge e_2$, so $\cR^g_{e_1,e_2}=fe_1\wedge e_2$ for some locally defined function $f$. Using this equation together with \eqref{rxi} we compute $s_g=4\alpha^2-2f$, whence $f=2\alpha^2-\tfrac12 s_g$. We thus get: 
\begin{equation}
\label{r12}
\cR^g_{e_1,e_2}=(2\alpha^2-\tfrac12 s_g)e_1\wedge e_2\, ,
\end{equation}

\noindent
which together with \eqref{rxi} implies \eqref{rg}. Finally, \eqref{ricg} follows by taking the trace. 
\end{proof}

\begin{lemma}
\label{lemma:Heisenberg}
Let $(M,g)$ be a non-flat compact Riemannian three-manifold carrying a unit length vector field $\xi$ satisfying Equation \eqref{eq:finalnablaxi} for some non-zero constant $\alpha$. If $s_g=-2 \alpha^2$, then $(M,g)$ is a compact quotient of the simply-connected Heisenberg group equipped with a left-invariant metric. 
\end{lemma}

\begin{proof}
If $s_g + 2\alpha^2 = 0$, Equation \eqref{ricg} implies that the Ricci tensor of $g$ reads:
 \[
\mathrm{Ric}^g = \begin{bmatrix}
2\alpha^2 & 0 & 0 \\
0 & -2\alpha^2 & 0 \\
0 & 0 & -2\alpha^2
\end{bmatrix}
\]
\noindent
and thus we recover the case considered in \cite[Proposition 4.12]{Moroianu:2023jof}, which proves that $(M,g)$ is a quotient of the Heisenberg group $\mathrm{H}_3$ equipped with a left-invariant metric. 
\end{proof}

\noindent
We now compute the curvature tensor of the connection $\nabla^{g,A}$ in \eqref{eq:finalnablaA}.

\begin{lemma}
\label{lemma:RA}
Let $(M,g)$ be a compact Riemannian three-manifold. Then the curvature tensor of the connection $\nabla^{g,A}$ given in Equation \eqref{eq:finalnablaA} reads:
\begin{equation*}
\cR^{g,A}_{X,Y} = \cR^g_{X,Y} + \alpha^2 (X\wedge Y) + 2\alpha\gamma\, \langle \ast_g\xi,X\wedge Y\rangle_g  \ast_g\xi\, ,
\end{equation*}

\noindent
where $\cR^g$ denotes the Riemann tensor of $g$. In particular $\mathcal{L}_{\xi}\cR^{g,A} = 0$, where $\mathcal{L}$ denotes the Lie derivative symbol.
\end{lemma}
 
\begin{proof}
We work with vector fields $X,Y,Z$ parallel with respect to $\nabla^g$ at a point. We compute:
\begin{eqnarray*}
 \nabla^{g,A}_X \nabla^{g,A}_Y Z &=&\nabla^{g,A}_X (\nabla^g_Y Z + \ast_g(A(Y)\wedge Z))\\ &=& \nabla^g_X \nabla^g_Y Z + \ast_g(A(\ast_g(A(X)\wedge Y))\wedge Z)
 + \ast_g(A(Y)\wedge \ast_g (A(X)\wedge Z))\\ &=&  \nabla^g_X \nabla^g_Y Z + Z\lrcorner \ast_g A(\ast_g(A(X)\wedge Y)) - A(Y) \lrcorner (A(X)\wedge Z)\, . 
\end{eqnarray*}

\noindent
Hence, skew-symmetrizing this expression in $X$ and $Y$, we obtain:
\begin{equation*}
\cR^{g,A}_{X,Y} = \cR^g_{X,Y} + \ast_g A(\ast_g(A(X)\wedge Y)) -  \ast_g A(\ast_g(A(Y)\wedge X))  - A(X)\wedge A(Y)\, . 
\end{equation*}

\noindent
Substituting $A = \alpha\, g + \gamma\, \xi\otimes \xi$ into the previous expression gives the desired result. On the other hand, since $\xi$ is a Killing vector field and $\mathcal{L}_{\xi}A=0$,  it follows directly that $\mathcal{L}_{\xi}\cR^{g,A} = 0$.
\end{proof}

\noindent
Using the previous proposition, we now compute the quadratic term $\cR^{g,A}\circ_g\cR^{g,A}$ that appears explicitly in the Einstein equation of the Heterotic soliton system.
\begin{lemma}
 \label{lemma:RARA}
Let $(M,g)$ be a Riemannian three-manifold carrying a unit length vector field $\xi$ satisfying \eqref{eq:finalnablaxi}. If $\nabla^{g,A}$ is the metric connection defined in \eqref{eq:finalnablaA}, then:
\begin{equation*}
\cR^{g,A} \circ_g \cR^{g,A} = (3\alpha^2-\tfrac12s_g+2\alpha\gamma)^2(g-\xi\otimes\xi)\, .    
\end{equation*}

\noindent
In particular, $(\cR^{g,A} \circ_g \cR^{g,A})(\xi) = 0$.
\end{lemma}
\begin{proof}
From Lemma \ref{lemma:RA} together with Equation \eqref{rg} we obtain: 
\begin{equation}
\label{rga}
\cR^{g,A}_{X,Y} = (3\alpha^2-\tfrac12 s_g+2\alpha\gamma)\, \langle \ast_g\xi,X\wedge Y\rangle_g  \ast_g\xi\,,\qquad\forall\ X,Y\in TM\, .
\end{equation}
The desired formula then immediately follows from the definition \eqref{rdrd} of $\cR^{g,A}\circ_g\cR^{g,A}$.
\end{proof}


\subsection{Heterotic solitons with generic reducible parallel torsion} 


Our first goal is to show that Heterotic solitons $(g, \frh,A)\in \Sol_{\kappa}(M)$ with auxiliary connection \eqref{eq:finalnablaA} necessarily have constant dilaton, that is, have $\varphi = \dd \log(\frh) = 0$. Since $\xi$ is a Killing vector field and by Lemma \ref{lemma:RA} preserves $\cR^{g,A}$, applying $\mathcal{L}_{\xi}$ to the dilaton equation, namely Equation \eqref{eq:HeteroticSystem3dII}, we obtain:
\begin{equation*}
\mathcal{L}_{\xi}(\delta^g\varphi + |\varphi|^2_g  - \frh^2) = 0\, . 
\end{equation*}

\noindent
Combining this equation with the $\mathcal{L}_{\xi}$-derivative of Equation \eqref{eq:dilatontrace} yields:
\begin{equation*}
\mathcal{L}_{\xi}(\delta^g\varphi + \tfrac{3}{2}\frh^2) = 0\, . 
\end{equation*}

\noindent
On the other hand, by Lemma \ref{lemma:Maxwellintegrated}, we have $\varphi = \dd \log(\frh)$, which plugged back into the previous equation gives:
\begin{equation*}
0 = \mathcal{L}_{\xi}(\delta^g\dd\log(\frh) + \tfrac{3}{2}\frh^2) = \Delta_g {\xi(\log(\frh))}  + 3\, \frh\, \xi(\frh)\, ,
\end{equation*}

\noindent
and thus:
\begin{equation*}
{\xi(\log(\frh))}\Delta_g \xi(\log(\frh))  = - 3\, \xi(\frh)^2\, . 
\end{equation*}

\noindent
Hence, integrating over $M$ we obtain $\xi(\frh) = 0$, that is, $\varphi(\xi) = 0$. In particular, $\mathcal{L}_{\xi}\varphi = 0$.

\begin{proposition}
\label{prop:varphi=0}
Let $(g,\frh,A)\in \Sol_{\kappa}(M)$ be a non-trivial compact Heterotic soliton with $A\in \Gamma(T^{\ast}M\otimes T^{\ast}M)$ given by \eqref{eq:finalA} for certain constants $\alpha\in \mathbb{R}$, $\gamma\in \mathbb{R}^{\ast}$. Then, $\alpha\neq 0$ and $\dd\frh = 0$. In particular, $\varphi=0$.
\end{proposition}

\begin{proof}
We plug $\xi$ into the Einstein equation, namely into the first equation of \eqref{eq:HeteroticSystem3dI}. By lemmas \ref{lemma:Riccixi} and \ref{lemma:RARA}, all terms in the Einstein equation when evaluated on $\xi$, are necessarily proportional to $\xi$, except for $\nabla^g_{\xi}\varphi$. We compute:
\begin{eqnarray*}
    \nabla^g_{\xi}\varphi = \nabla^g_{\varphi}\xi = \alpha \ast_g (\xi \wedge \varphi)\, ,
\end{eqnarray*}

\noindent
which is clearly orthogonal to $\xi$. If $\alpha = 0$, by Lemma \ref{lemma:Riccixi} we have $\mathrm{Ric}^g(\xi) = 0$, and therefore the evaluation of the Einstein equation on $\xi$ reduces to $\frh=0$, a contradiction. Hence, $\alpha\neq 0$ and $\xi \wedge \dd\frh = 0$. Therefore, by the previous discussion, also $\xi(\frh) = 0$. Consequently, $\dd\frh = 0$, or, equivalently, $\varphi=0$.
\end{proof}

\noindent
By Proposition \ref{prop:varphi=0}, plugging $\varphi = 0$ in Equation \eqref{eq:dilatontrace}, we obtain that the scalar curvature of a non-trivial compact Heterotic soliton with generic reducible parallel torsion must be a negative constant:

\begin{corollary}
\label{cor:scalarh}
Let $(g,\frh,A)\in \Sol_{\kappa}(M)$ be a non-trivial compact Heterotic soliton with parallel contorsion tensor $A\in \Gamma(T^{\ast}M\otimes T^{\ast}M)$ given by \eqref{eq:finalA} for certain constants $\alpha\in \mathbb{R}$, $\gamma\in \mathbb{R}^{\ast}$. Then:
\begin{equation*}
s_g = - \tfrac{1}{2} \frh^2\, . 
\end{equation*}

\noindent
In particular, the scalar curvature is constant.
\end{corollary}

\begin{lemma}
\label{lemma:NoYM}
Let $(g,\frh,A)\in \Conf(M)$ be a triple with $A\in \Gamma(T^{\ast}M\otimes T^{\ast}M)$ given by \eqref{eq:finalA} for certain constants $\alpha\in \mathbb{R}$, $\gamma\in \mathbb{R}^{\ast}$. If $(g,\frh,A)$ satisfies the first equation in \eqref{eq:HeteroticSystem3dI} as well as Equation \eqref{eq:HeteroticSystem3dII}, then:
\begin{equation*}
\nabla^{g,A}\cR^{g,A} = 0\, ,
\end{equation*}

\noindent
and consequently the Yang-Mills equation, namely the second equation in \eqref{eq:HeteroticSystem3dI}, is automatically satisfied.
\end{lemma}

\begin{proof}
We first observe that all the results obtained thus far for Heterotic solitons rely only on the Einstein and dilaton equations of the Heterotic soliton system, and not on the Yang-Mills equation being satisfied. Hence, the result follows from $\nabla^{g,A}\xi= 0$, $\dd s_g = 0$, and the fact that Lemma \ref{lemma:Riccixi} implies that:
\begin{equation*}
 \mathrm{Ric}^g = (\tfrac12s_g - \alpha^2) g + (3 \alpha^2 - \tfrac12s_g) \xi\otimes \xi \, ,
\end{equation*}

\noindent
which is clearly parallel with respect to $\nabla^{g,A}$. 
\end{proof}

\begin{theorem}
\label{thm:Sasaki}
Let $M$ be a compact oriented three-manifold. Up to a double cover, $M$ admits a non-trivial Heterotic soliton $(g,\frh,D)$ with generic reducible parallel torsion if and only if $(M,g)$ is the compact quotient of the Heisenberg group equipped with a left-invariant metric.
\end{theorem}

\begin{proof}
Let $(g,\frh,D)$ be a Heterotic soliton with generic reducible parallel torsion. By Proposition \ref{prop:varphi=0}, $\frh$ is constant and consequently $\varphi =0$, from which we readily compute $\mathrm{Ric}^{g,H}=\mathrm{Ric}^{g}-\tfrac12\frh^2\, g$ upon use of Equation \eqref{eq:Ricgh}. Hence, by Lemma \ref{lemma:RARA}, together with Equation \eqref{ricg}, and Corollary \ref{cor:scalarh}, the Einstein equation, that is, the first equation in the system \eqref{eq:motionsugratorsion}, is equivalent to:
\begin{equation*}
0=(\tfrac12s_g - \alpha^2) g + (3 \alpha^2 - \tfrac12s_g) \xi\otimes \xi+s_g\, g+\kappa(3\alpha^2-\tfrac12s_g+2\alpha\gamma)^2(g-\xi\otimes\xi)\, ,
\end{equation*}
which in turn is equivalent to the system:
\begin{equation}\label{sys}
\begin{cases}
0=\tfrac32s_g - \alpha^2+\kappa(3\alpha^2-\tfrac12s_g+2\alpha\gamma)^2\\
0=3 \alpha^2 - \tfrac12s_g-\kappa(3\alpha^2-\tfrac12s_g+2\alpha\gamma)^2
\end{cases}
\end{equation}
As $\alpha\ne 0$ by Proposition \ref{prop:varphi=0}, the system \eqref{sys} is further equivalent to:
\begin{equation}
\label{eq:cases}
    \begin{cases}
        s_g=& - 2\alpha^2\\
         1\ =&\kappa(2\alpha+\gamma)^2
    \end{cases}
\end{equation} 

\noindent
Hence, by Lemma \ref{lemma:Heisenberg} $(M,g)$ is a compact quotient of the Heisenberg group equipped with a left-invariant metric. \medskip

\noindent
For the converse, assume that $(M,g)$ is the compact quotient of the Heisenberg group equipped with a left-invariant metric. Modulo a double cover, denote by $\xi$ the simple eigenvector of the Ricci tensor of $g$, which satisfies \cite[Proposition 4.14]{Moroianu:2021kit}:
\begin{eqnarray}
\label{eq:Heisenberg}
\mathrm{Ric}^g(\xi) = -s_g \xi\, , \qquad \nabla^g\xi = \sqrt{-\tfrac12{s_g}} \ast_g\xi\, .
\end{eqnarray}

\noindent
Define:
\begin{equation*}
\frh:=  \sqrt{-2 s_g}\, , \qquad D_X Y := \nabla^g_X Y + \sqrt{-\tfrac12{s_g}} \ast_g (X\wedge Y)  + (\tfrac{1}{\sqrt{\kappa}} -  \sqrt{-2s_g})\, \xi(X) \ast_g (\xi\wedge Y)\, .
\end{equation*}

\noindent
We claim that $(g,\frh,D)$ is an Heterotic soliton with non-trivial parallel torsion. The fact that the torsion of $D$ is parallel follows from the fact that $s_g$ is constant and $\xi$ is parallel:
\begin{equation*}
D_X \xi = \nabla^g_X \xi + \sqrt{-\tfrac12{s_g}} \ast_g (X\wedge \xi) = 0\, ,
\end{equation*}

\noindent
where we have used the second equation in \eqref{eq:Heisenberg}. Since $s_g\neq 0$, the connection $D$ is a connection with generic reducible non-trivial torsion. Since $\frh$ is constant, the second equation in the Heterotic soliton system \eqref{eq:HeteroticSystem3dII} is automatically satisfied with $\varphi=0$. This, together with Lemma \ref{lemma:NoYM}, implies that $(g,\frh,D)$ satisfies the Yang-Mills equation of the Heterotic soliton system, namely the second equation in \eqref{eq:HeteroticSystem3dI}. The dilaton equation, namely the first equation in \eqref{eq:HeteroticSystem3dII} is satisfied by the choice $\frh=  \sqrt{-2 s_g}$. Finally, the Einstein equation of the system, namely the first equation in \eqref{eq:HeteroticSystem3dI}, is satisfied by construction since we have chosen $\frh$ and $D$ such that both equations in \eqref{eq:cases} are satisfied. 
\end{proof}

\begin{remark}
Note that, in the proof of the previous theorem, we could have chosen the opposite sign roots when isolating for $\alpha$ and $\gamma$, resulting in a different combination of signs in the definition of $D$. 
\end{remark}


\section{Skew-symmetric parallel torsion}
\label{sec:Skewcontorsion}


In this section, we consider three-dimensional Heterotic compact solitons with totally skew-symmetric parallel torsion, which we assume to be non-zero. This covers point (B) in Remark \ref{remark:FinalCases} in the case of non-zero torsion. Note that this case was excluded in the previous section due to the condition $\gamma\neq 0$. As explained in the introduction, the remaining case of vanishing torsion is considerably more difficult and is therefore left open. Hence, we consider tuples $(g,\varphi,\frh,D)\in \Conf(M)$ for which the auxiliary connection has totally skew-symmetric non-zero parallel torsion and can be thus written as follows:
\begin{eqnarray*}
D_X = \nabla^{g,\alpha}_X := \nabla^g_X + \alpha \ast_g X\, , \qquad \forall \ X\in TM
\end{eqnarray*}

\noindent
for a constant $\alpha\in \mathbb{R}^*$. Hence, we will denote this class of tuples simply by $(g,\frh,\alpha)\in \Conf(M)$. The curvature of $\nabla^{g,\alpha}$ follows directly as a particular case of Lemma \ref{lemma:RA}:
\begin{equation}
\label{eq:Rf}
 \cR^{g,\alpha}_{X,Y} = \cR^g_{X,Y} + \alpha^2 X\wedge Y\, , \qquad \forall \ X, Y\in TM\, .
\end{equation}

\noindent
From this, together with Equation \eqref{eq:appR3d}, we obtain:
\begin{eqnarray}
\label{eq:RfRf}
 \cR^{g,\alpha} \circ_g \cR^{g,\alpha} = - \mathrm{Ric}^g\circ_g\mathrm{Ric}^g + (s_g - 2\alpha^2) \mathrm{Ric}^g + (\vert\mathrm{Ric}^g\vert^2_g - \tfrac12{s_g} + 2\alpha^4) g\, ,
\end{eqnarray}

\noindent
which we use in the proof of the following proposition.

\begin{proposition}
\label{prop:skew3d}
A triple $(g,\frh,\alpha)\in \Conf(M)$ satisfies the Heterotic soliton system with totally skew-symmetric parallel torsion if and only if:
\begin{align*}
\nonumber  -\kappa\mathrm{Ric}^g \circ_g \mathrm{Ric}^g +(1 +  \kappa s_g - 2\kappa\alpha^2) \mathrm{Ric}^{g}\hskip4cm&\\ + \left (\kappa \vert\mathrm{Ric}^g \vert_g^2-\tfrac12\kappa{s_g^2} - \tfrac1{2}{\frh^2} + 2\kappa\alpha^4\right)  g  +\nabla^{g}\dd\log(\frh)&= 0 \, , \\
(\dd_{\nabla^g}\mathrm{Ric}^g)(X) - 3\alpha \ast_g \mathrm{Ric}^g_0(X) +\cR^g_{\dd\log(\frh),X} + \alpha^2 \dd\log(\frh)\wedge X& = 0  \, ,\\
s_g - 3\, \delta^g \dd\log(\frh) - 2 \vert \dd\log(\frh)\vert^2 +\tfrac{1}{2} \frh^2 &= 0 \, ,
\end{align*}

\noindent
for every $X\in TM$, where $\mathrm{Ric}^g_0 := \mathrm{Ric}^g - \tfrac13 s_g g$ denotes the traceless Ricci tensor.  
\end{proposition}

\begin{proof}
 The first equation in the proposition corresponds to the first equation in \eqref{eq:HeteroticSystem3dI}, namely the Einstein equation of the system, after plugging Equation \eqref{eq:Ricgh}, Equation \eqref{eq:RfRf}, and substituting $\varphi = \dd\log(\frh)$. \medskip

 \noindent
 The second equation in the proposition corresponds to the second equation \eqref{eq:HeteroticSystem3dI}, namely the Yang-Mills equation of the system. Using a vector field $X$ and a local orthonormal frame parallel at a given point, we compute using the Bianchi identity together with Equation \eqref{eq:appR3d}:
 \begin{align*}
 (\nabla^{g,\alpha\ast} \cR^{g,\alpha})(X) & = -(\nabla^{g,\alpha}_{e_i} \cR^{g,\alpha})_{e_i,X} = -(\nabla^{g,\alpha}_{e_i} \cR^{g})_{e_i,X} =  -\nabla^{g,\alpha}_{e_i} \cR^{g}_{e_i,X} + \alpha\,\cR^{g}_{e_i,\ast_g(e_i\wedge X)}\\
& =  -\nabla^{g}_{e_i} \cR^{g}_{e_i,X} -\alpha\, (\ast_g e_i) (\cR^g_{e_i,X})+ \alpha\,\cR^{g}_{e_i,\ast_g(e_i\wedge X)}\\
& =  -\nabla^{g}_{e_i} \cR^{g}_{e_i,X} + \alpha\, (\cR^g_{e_i,X}) (\ast_g e_i)  + \alpha\,\cR^{g}(e_i,\ast_g(e_i\wedge X))\\
&  = \dd^{\nabla^g}\mathrm{Ric}^g - 3\alpha\,\ast_g \mathrm{Ric}^g_0(X)\, .  
 \end{align*}

\noindent
From this, the result follows after using Equation \eqref{eq:Rf}. The third equation in the statement is Equation \eqref{eq:dilatontrace}, which is equivalent to the first equation of \eqref{eq:HeteroticSystem3dII}.
\end{proof}

\begin{remark}
\label{remark:f}
By subtracting the third equation in Proposition \ref{prop:skew3d} from the trace of the Einstein equation, namely the first equation in Proposition \ref{prop:skew3d}, we obtain the identity:
\begin{equation*}
2\kappa \vert \mathrm{Ric}^g_0\vert^2_{g} + 2 \vert \dd\log(\frh) \vert^2_g - 2 \frh^2 + \tfrac16{\kappa} (s_g - 6\alpha^2)^2 + 2\delta^g\dd\log(\frh) = 0\, ,
\end{equation*}

\noindent
which will be used below.
\end{remark}

\noindent
Our first task is to prove that $\frh$ is constant. For this, we first need the following lemmata.

\begin{lemma}
\label{lemma:sgnof}
Let $(g,\frh,\frf)\in \Sol_{\kappa}(M)$ be a compact non-trivial Heterotic soliton with completely skew-symmetric parallel torsion. Then, $s_g \neq 6 \alpha^2$ as functions on $M$.
\end{lemma}

\begin{proof}
Suppose $s_g = 6 \alpha^2$ everywhere on $M$. Then, Remark \ref{remark:f} gives:
\begin{equation*}
2\kappa \vert \mathrm{Ric}^g_0\vert^2_{g} + 2 \vert \dd\log(\frh) \vert^2_g - 2 \frh^2  + \delta^g\dd\log(\frh) = 0\, . 
\end{equation*}

\noindent
Adding to this equation the third equation of Proposition \ref{prop:skew3d} multiplied by 4, we obtain:
\begin{equation*}
2\kappa \vert \mathrm{Ric}^g_0\vert^2_{g} +   24 \alpha^2  - 6  \vert \dd\log(\frh)\vert^2   -10 \delta^g\dd\log(\frh) = 0\, . 
\end{equation*}

\noindent
Evaluating this expression at a maximum of $\frh$ we obtain $\alpha = 0$, a contradiction. 
\end{proof}

\begin{lemma}
\label{lemma:Ricci0}
Let $(g,\frh,\frf)\in \Sol_{\kappa}(M)$ be a compact Heterotic soliton with completely skew-symmetric parallel torsion. Then, $\mathrm{Ric}^g_0 (\dd\frh) = 0$ and the following equation is satisfied:
\begin{equation*}
\tfrac{1}{2} \dd s_g = (\tfrac13{s_g} - 2\alpha^2)\,\dd\log(\frh) \, . 
\end{equation*}
\end{lemma}

 \begin{proof}
We project the Yang-Mills equation, namely the second equation in Proposition \ref{prop:skew3d}, to $\wedge^1 M$. In other words, for a local orthonormal frame $\{e_i\}$ parallel at the point where the computation is done, we take $X=e_i$, contract with $e_i$ and sum over $i$, obtaining:
\begin{align}\label{eq:projectYM1}
0&= e_i \lrcorner (e_j\wedge (e_i \lrcorner\nabla^g_{e_j} \mathrm{Ric}^g)) +  \mathrm{Ric}^g (\dd\log(\frh)) - 2\alpha^2 \dd\log(\frh) \nonumber \\ 
& =   - \tfrac{1}{2} \dd s_g   + \mathrm{Ric}^g (\dd\log(\frh)) - 2 \alpha^2 \dd\log(\frh)\, ,
\end{align}
where in the last equality we used the contracted Bianchi identity $\nabla^{g\ast}\mathrm{Ric}^g=-\tfrac12\dd s_g$. Evaluating the Yang-Mills equation on $\dd\log(\frh)$, we obtain:
\begin{eqnarray*}
(\dd_{\nabla^g}\mathrm{Ric}^g)(\dd\log(\frh)) = 3\alpha \ast_g(\mathrm{Ric}^g_0(\dd\log(\frh)))\, . 
\end{eqnarray*}
We then compute:
\begin{eqnarray*}
 (\dd_{\nabla^g}\mathrm{Ric}^g)(\dd\log(\frh))& =& e_j \wedge (\dd\log(\frh)\lrcorner\nabla^g_{e_j} \mathrm{Ric}^g) \\
 & = &\dd (\mathrm{Ric}^g(\dd\log(\frh))) - e_j \wedge \mathrm{Ric}^g(\nabla^g_{e_j} \dd\log(\frh)) = 0 \, ,   
\end{eqnarray*} 

\noindent
where we have used Equation \eqref{eq:projectYM1} together with the fact that, by the Einstein equation in Proposition \ref{prop:skew3d}, the term $\mathrm{Ric}^g(\nabla^g_{e_j} \dd\log(\frh))$ is proportional to a symmetric tensor evaluated in $e_j$. Hence, we obtain:
\begin{equation*}
 \mathrm{Ric}^g_0(\dd\log(\frh)) = 0   
\end{equation*}

\noindent
upon use of $\alpha \neq 0$. Plugging this equation back into \eqref{eq:projectYM1}, we obtain the equation in the statement.
\end{proof}

\begin{proposition}
\label{prop:constantf} 
Let $(g,\frh,\alpha)\in \Sol_{\kappa}(M)$ be a compact Heterotic soliton with completely skew-symmetric parallel torsion. Then, $\dd \frh = 0$ and $2s_g = -\frh^2$. In particular, the scalar curvature of $g$ is a strictly negative constant.
\end{proposition}

\begin{proof}
We first observe that the Yang-Mills equation, namely the second equation in Proposition \ref{prop:skew3d}, can be written as follows:
\begin{equation}
\label{eq:YMauxiliar}
(\dd_{\nabla^g}\mathrm{Ric}^g)(X) - 3\alpha \ast_g \mathrm{Ric}^g_0(X) + \dd\log(\frh) \wedge ( \alpha^2 X - \tfrac16{s_g} X - \mathrm{Ric}^g_0(X)) = 0\, ,
\end{equation}

\noindent
for every $X\in TM$. Here we have used the identity:
\begin{equation*}
    \cR^{g}_{\dd\frh,X} = -\tfrac16{s_g}\dd\frh \wedge X + \mathrm{Ric}^g_0(X)\wedge \dd\frh+X\wedge\mathrm{Ric}^g_0(\dd\frh)\, ,
\end{equation*}

\noindent
in combination with Lemma \ref{lemma:Ricci0}. We proceed by computing the divergence of the Yang-Mills equation, written as in Equation \eqref{eq:YMauxiliar}, by using an orthonormal frame $e_i$ parallel at a point. For the first term, we have:
\begin{equation*}
(\nabla^g_{e_i}\dd_{\nabla^g}\mathrm{Ric}^g)(e_i) = \nabla^g_{e_i}(\dd_{\nabla^g}\mathrm{Ric}^g(e_i)) = e_j \wedge ((\cR^g_{e_i,e_j}\mathrm{Ric}^g)(e_i)) = 0\, . 
\end{equation*}

\noindent
Regarding the second term, we obtain (using again the contacted Bianchi identity):
\begin{eqnarray*}
\nabla^g_{e_i} (\ast_g \mathrm{Ric}^g_0(e_i)) = -\ast_g \nabla^{g\ast}\mathrm{Ric}^g_0 = \tfrac{1}{6}\ast_g \dd s_g\, . 
\end{eqnarray*}

\noindent
For the third term, we compute:
\begin{eqnarray*}
& \nabla^{g}_{e_i} (\dd\log(\frh) \wedge ( \alpha^2 e_i - \tfrac16{s_g} e_i - \mathrm{Ric}^g_0(e_i))) =  \mathrm{Ric}^g_0(e_i) \wedge \nabla^{g}_{e_i} \dd\log(\frh) + \tfrac{1}{6} \dd s_g\wedge \dd \log(\frh)\\
& + \dd\log(\frh) \wedge \nabla^{g\ast}\mathrm{Ric}^g_0 = \tfrac{1}{3} \dd s_g\wedge \dd \log(\frh)\, ,
\end{eqnarray*}

\noindent
where we have used that:
\begin{equation*}
  \mathrm{Ric}^g_0(e_i) \wedge \nabla^{g}_{e_i} \dd\log(\frh)  = 0
\end{equation*}

\noindent
by virtue of isolating the term $\nabla^{g}_{e_i} \dd\log(\frh)$ in the Einstein equation (first equation in Proposition \ref{prop:skew3d}) evaluated on $e_i$. Hence, all in all, we obtain:
\begin{eqnarray*}
  \tfrac{1}{2}\alpha\, \dd s_g  = \tfrac{1}{3} \ast_g (\dd s_g\wedge \dd \log(\frh))\, . 
\end{eqnarray*}

\noindent
Due to $\alpha\neq 0$, this equation is equivalent to $\dd s_g = 0$. By Lemma \ref{lemma:Ricci0}, this implies in turn:
\begin{equation*}
(s_g - 6 \alpha^2)\,\dd\log(\frh)  = 0\, . 
\end{equation*}

\noindent
Since, by Lemma \ref{lemma:sgnof}, we have $s_g \neq 6\alpha^2$ as functions on $M$, and now both $s_g$ and $\alpha$ are constant functions, it follows that $\dd\frh = 0$ and, consequently, the third equation in Proposition \ref{prop:skew3d} reduces to $2s_g = - \frh^2$.
\end{proof}

\noindent
We are now ready to present the final classification results for this section.

\begin{theorem}
\label{thm:skewtorsion}
Let $M$ be an oriented compact three-manifold. Up to a double cover, $M$ admits a Heterotic soliton $(g,\frh,\alpha)$ with non-trivial completely skew-symmetric parallel torsion if and only if $(M,g)$ is either isometric to a compact quotient of the Heisenberg group equipped with a left-invariant metric of scalar curvature $2\kappa  s_g = -1$ or isometric to a compact hyperbolic three-manifold of scalar curvature $\kappa s_g\in (-24 , 0)$.
\end{theorem}

\begin{proof}
To prove the \emph{only if} condition, consider a compact Heterotic soliton $(g,\frh,\alpha)\in \Sol_{\kappa}(M)$ with completely skew-symmetric non-zero parallel torsion. By Proposition \ref{prop:constantf}, $\frh$ and $\alpha$ are both constant, the scalar curvature is given by $s_g = -\tfrac12\frh^2$, and by Remark \ref{remark:H0}, $\frh$ is non-vanishing. Then, the Einstein equation in Proposition \ref{prop:skew3d} becomes an algebraic equation of second order for the Ricci tensor $\mathrm{Ric}^g$, generalizing the type of equation appearing in \cite[Proposition 4.9]{Moroianu:2023jof}. Hence, either $g$ is Einstein or $\mathrm{Ric}^g$ has two different constant eigenvalues. If $g$ is Einstein, then the Heterotic soliton system in Proposition \ref{prop:skew3d} reduces to:
\begin{eqnarray*}
-  \tfrac19 \kappa \,{s_g^2} + (1 + \kappa s_g - 2\kappa\, \alpha^2  )\tfrac13{s_g}  +  ( \tfrac13 \kappa\,{s_g^2} -\tfrac12{\kappa}\, s_g^2  - \tfrac{1}{2} \frh^2 + 2\kappa\,\alpha^4 )   = 0\, ,
\end{eqnarray*}
which taking into account that $2 s_g = -\frh^2$, simplifies to:
\begin{equation}
\label{eq:constrainthf}
\kappa (\frh^2 + 12 \alpha^2)^2 = 48 \frh^2\, . 
\end{equation}

\noindent
For any solution of this equation, $(M,g)$ is a compact Einstein manifold of strictly negative curvature and therefore a compact hyperbolic three-manifold. Furthermore, using that $\alpha\neq 0$, it follows from the previous algebraic equation that $\kappa\, \frh^2 \in (0,48)$, whence $\kappa\, s_g \in (-24,0)$.\medskip

\noindent
Suppose now that $g$ is not Einstein but instead has two different constant eigenvalues, say $\mu_1 , \mu_2 \in \mathbb{R}$, with $\mu_1$ being the double eigenvalue. Then, denoting by $\xi$ a unit vector field corresponding to the simple eigenvalue (which always exists on $M$ or some double cover of it), we can write:
\begin{equation*}
\mathrm{Ric}^g = \mu_1 g+(\mu_2-\mu_1)\xi\otimes\xi,\qquad \mathrm{Ric}^g_0=(\mu_2-\mu_1)\left(-\tfrac13g+\xi\otimes\xi\right)\, . 
\end{equation*}

\noindent
Using this equation, we first impose the Yang-Mills equation of Proposition \ref{prop:skew3d} and then the Einstein equation. Since $\alpha$ and $\frh$ are constant, the Yang-Mills equation reduces to:
\begin{equation*}
 (\dd^{\nabla^g}\mathrm{Ric}^g)(X) - 3 \alpha \ast_g\mathrm{Ric}^g_0(X)  = 0\, , \qquad \forall\ X\in TM\, . 
\end{equation*}

\noindent
We compute:
\begin{eqnarray*}
\dd^{\nabla^g}\mathrm{Ric}^g=(\mu_2-\mu_1)(\dd\xi\otimes\xi+(e_i\wedge\xi)\otimes\nabla^g_{e_i}\xi)    \, ,
\end{eqnarray*}

\noindent
and thus the Yang-Mills equation becomes:
\begin{eqnarray*}
\xi(X)\,\dd\xi +g(\nabla^g_{e_i}\xi,X)\,e_i\wedge\xi - 3 \alpha \ast_g (-\tfrac13 X + \xi(X)\,\xi ) = 0\,,\qquad\forall\ X\in TM\, .    
\end{eqnarray*}

\noindent
For $X=\xi$ this gives $\dd\xi = 2 \alpha \ast_g \xi$. Plugging this back into the previous equation, we obtain:
\begin{eqnarray*}
g(\nabla^g_{e_i}\xi,X)\,e_i\wedge\xi + \alpha \ast_g (X - \xi(X)\,\xi ) = 0 \, ,\qquad\forall \ X\in TM\, .    
\end{eqnarray*}

\noindent
Taking the interior product of this equation with $\xi$ and using again $\dd\xi=2\alpha \ast_g\xi$ yields:
\begin{eqnarray*}
\alpha \ast_g (\xi \wedge X)&=&g(\nabla^g_\xi\xi,X)\xi-g(\nabla^g_{e_i}\xi,X)\,e_i \\
&=&\dd\xi(\xi,X) \xi-g(\nabla^g_X\xi,e_i)\,e_i-\dd\xi(e_i,X)\,e_i\\
& =&-\nabla^g_X\xi + 2\alpha \ast_g (\xi\wedge X)\, ,    
\end{eqnarray*}

\noindent
implying that $\nabla^g_X\xi= \alpha \ast_g (\xi\wedge X)$ for every $X\in TM$. In particular, $\xi$ is Killing. The Bochner formula gives $\mathrm{Ric}^g(\xi) = 2\alpha^2\xi$, and therefore:
\begin{equation*}
 2\mu_2 = 4\alpha^2=-2s_g=-2(2\mu_1+\mu_2)\, .
\end{equation*}

\noindent
In particular, $\mu_1+\mu_2=0$, and hence we can set $\mu := \mu_1 = - \mu_2$. On the other hand, by Equation \eqref{eq:dilatontrace}, we have $2s_g = -\frh^2$, and therefore, we conclude:
\begin{eqnarray*}
\frh^2 = 4\alpha^2\, .
\end{eqnarray*}

\noindent
Plugging this expression into the Einstein equation of Proposition \ref{prop:skew3d}, together with:
\[
\mathrm{Ric}^g = \begin{bmatrix}
\mu & 0 & 0 \\
0 & - \mu & 0 \\
0 & 0 & - \mu
\end{bmatrix} 
\]

\noindent
we obtain:
\begin{equation}
\label{eq:algebraicHeisenberg}
 -  \kappa \, \mu^2 + (1  - 4\kappa  \, \alpha^2  )\mu +  ( 3\kappa \mu^2   - 2\alpha^2 )    = 0\, ,\qquad \mu = -2\alpha^2\, .
\end{equation}

\noindent
The general solution to this equation is:
\begin{equation*}
4 \kappa \alpha^2 = 1\, , \qquad \mu = \tfrac{1}{2\kappa}\, , 
\end{equation*}

\noindent
which implies:
\begin{equation*}
    \kappa s_g = -\tfrac{1}{2}\, .
\end{equation*}

\noindent
Hence, we precisely recover the case considered in \cite[Theorem 4.15]{Moroianu:2023jof}, which proves that in this case $(M,g)$ is isometric to the compact quotient of the Heisenberg group equipped with a left-invariant metric of scalar curvature $2\kappa s_g = -1$. \medskip

\noindent
To prove the \emph{if} direction, consider first $(M,g)$ to be a compact hyperbolic three-manifold of scalar curvature $\kappa s_g\in (-24 , 0)$. We set:
\begin{equation*}
\frh := \sqrt{-2 s_g}\, , \qquad D_X := \nabla^g_X + \alpha \ast_g X\, ,
\end{equation*}

\noindent
where we choose $\alpha$ such that equation \eqref{eq:constrainthf} is satisfied for the previous choice of $\frh$. The fact that $s_g \in (-24 , 0)$ guarantees that a solution $\alpha\in\mathbb{R}^*$ exists. Then, $(g,\frh,\alpha)$ is a Heterotic soliton. Indeed, the dilaton equation follows directly from the previous choice of $\frh$, since $\varphi = 0$. The Maxwell equation for $D$ follows directly since $\dd_{\nabla^g}\mathrm{Ric}^g = 0$ and $\mathrm{Ric}^g_0 = 0$ as a consequence of $g$ being Einstein. The Einstein equation also follows by reversing the arguments that lead to Equation \eqref{eq:constrainthf}. 

Assume now that $(M,g)$ is the compact quotient of the Heisenberg group equipped with a left-invariant metric $g$ of scalar curvature $2\kappa  s_g = -1$. Modulo a double cover, denote by $\xi$ the simple eigenvector of the Ricci tensor of $g$, which satisfies \cite[Proposition 4.14]{Moroianu:2021kit}:
\begin{eqnarray}
\label{eq:HeisenbergII}
\mathrm{Ric}^g = - \tfrac{1}{2\kappa} g + \tfrac{1}{\kappa}\xi\otimes\xi\, , \qquad \nabla^g\xi = \tfrac{1}{2\sqrt{\kappa}} \ast_g\xi\, .
\end{eqnarray}

\noindent
where we have used the relation $2\kappa  s_g = -1$. We set:
\begin{equation*}
\frh:=  \tfrac{1}{\sqrt{\kappa}}\, , \qquad D_X Y := \nabla^g_X Y + \tfrac{1}{2\sqrt{\kappa}} \ast_g (X\wedge Y)\, , \qquad \forall\ X, Y \in \mathfrak{X}(M) \, .
\end{equation*}

\noindent
We claim that $(g,\frh,D)$ is an Heterotic soliton with non-trivial parallel skew-symmetric torsion. The fact that the torsion of $D$ is parallel follows from the fact that $\kappa > 0$ is constant.  Hence, $\frh$ is also constant, and since $2\kappa  s_g = -1$, it follows from the definition of $\frh$ that $2s_g = -\frh^2$ and consequently the third equation in Proposition \ref{prop:skew3d} is satisfied. Thanks again to $\frh$ being constant, the second equation in Proposition \ref{prop:skew3d} reduces to:
\begin{equation*}
(\dd^{\nabla^g}\mathrm{Ric}^g)(X)  - \tfrac{3}{2\sqrt{\kappa}} \ast_g (\mathrm{Ric}^g_0(X)) = 0\, , \qquad \forall \ X\in TM  \, ,  
\end{equation*}

\noindent
and is satisfied via a direct computation as a consequence of Equation \eqref{eq:HeisenbergII}. Finally, the first equation in Proposition \ref{prop:skew3d} is shown to be satisfied by tracing back the computations that lead to Equation \eqref{eq:algebraicHeisenberg}.
\end{proof}
  
\noindent
In the previous theorem, we have assumed that $D$ has non-vanishing skew-symmetric torsion, which amounts $\alpha \neq 0$. However, if we formally set $\alpha = 0$ in the previous theorem, we obtain $D=\nabla^g$ and $\kappa s_g = -24$, and a quick computation reveals that:
\begin{equation*}
(g, \frh = \sqrt{\tfrac{48}{\kappa}},\nabla^g)\, ,
\end{equation*}

\noindent
where $g$ is a hyperbolic metric of scalar curvature $\kappa s_g = -24$, is indeed a Heterotic soliton, albeit of \emph{vanishing} torsion. However, this does not mean that \emph{all} Heterotic solitons with vanishing torsion are of this form or have $\frh$ constant. As mentioned earlier, this case is significantly more difficult and will be considered elsewhere.

\phantomsection
\bibliographystyle{JHEP}

\begin{thebibliography}{100}
	
 

\bibitem{BRI}{E. Bergshoeff and M. de Roo, {\sl Supersymmetric Chern-Simons Terms in Ten-dimensions}, Phys. Lett. B {\bf 218} (2), 210--215 (1989).}

\bibitem{BRII}{E. Bergshoeff and M. de Roo, {\sl The Quartic Effective Action of the Heterotic String and Supersymmetry}, Nucl. Phys. B {\bf 328} (2), 439--468 (1989).}

\bibitem{Marisa}{M. Fernández, S. Ivanov, L. Ugarte, and R. Villacampa, {\sl Compact supersymmetric solutions of the heterotic equations of motion in dimensions 7 and 8}, Adv. Theor. Math. Phys. {\bf 15} (2), 245--284 (2011).} 

\bibitem{Friedrich:2001nh}{T.~Friedrich and S.~Ivanov, {\sl Parallel spinors and connections with skew symmetric torsion in string theory}, Asian J. Math. \textbf{6}, 303--336 (2002).}
 
\bibitem{GarciaFernandez}{M. Garc\'ia-Fern\'andez and J. Streets,  {\sl Lectures on the Strominger system}, Travaux Mathematiques Vol. XXIV, 7--61 (2016).}

\bibitem{FernandezStreetsLibro}{M. Garc\'ia-Fern\'andez,  {\it Generalized Ricci Flow}, AMS University Lecture Series, 2021.}

\bibitem{Gillard:2004xq}{J.~Gillard, U.~Gran and G.~Papadopoulos, {\sl The Spinorial geometry of supersymmetric backgrounds}, Class. Quant. Grav. \textbf{22}, 1033--1076 (2005).}

\bibitem{Hull}{C. Hull, {\sl Compactifications of the Heterotic Superstring}, Phys. Lett. B {\bf 191}, 357--364 (1986).}

\bibitem{Ivanov}{S. Ivanov, {\sl Connection with torsion, parallel spinors and geometry of Spin(7) manifolds}, Math. Res. Lett. {\bf 11} (2-3), 171--186 (2004).}

\bibitem{Melnikov:2014ywa}{I.~V.~Melnikov, R.~Minasian and S.~Sethi, {\sl Heterotic fluxes and supersymmetry}, JHEP \textbf{6}, 174 (2014).}

\bibitem{Moroianu:2021kit}{A.~Moroianu, \'A.~Murcia and C.~S.~Shahbazi, {\sl Heterotic solitons on four-manifolds}, New York J. Math. {\bf 28}, 1463--1497 (2022).}

\bibitem{Moroianu:2023jof}{A.~Moroianu, {\'A}.~Murcia and C.~S.~Shahbazi, {\sl The Heterotic-Ricci Flow and Its Three-Dimensional Solitons}, J. Geom. Anal. \textbf{34} (5), 122 (2024).}

\bibitem{Oliynyk}{T. Oliynyk, V. Suneeta, and E. Woolgar, {\sl A gradient flow for worldsheet nonlinear sigma models}, Nuclear Phys. B {\bf 739} (3), 441--458 (2006).}

\bibitem{Papadopoulos:2024tgs}{G.~Papadopoulos, {\sl Scale and conformal invariance in heterotic {\ensuremath{\sigma}}-models}, JHEP \textbf{2025}, 112 (2025).}

\bibitem{Picard}{S. Picard, {\sl The Strominger system and flows by the Ricci tensor}, Surveys in Differential Geometry {\bf 27} (1), 103--145 (2022).}
 
\bibitem{Polchinski:1998rq}{J.~Polchinski, {\it String theory. Vol. 1: An introduction to the bosonic string}, Cambridge Monographs on Mathematical Physics, 1998.}

 \bibitem{Raghunathan}{M. S. Raghunathan, {\it Discrete Subgroups of Lie Groups}, Ergebnisse der Mathematik und ihrer Grenzgebiete, Springer-Verlag, {\bf 68}, 1972.}

 \bibitem{Strominger}{A. Strominger, {\sl Superstrings with torsion}, Nucl. Phys. B {\bf 274} (2), 253--284 (1986).}

\bibitem{StreetsSoliton}{J. Streets, {\sl Classification of solitons for pluriclosed flow on complex surfaces}, Math. Ann. {\bf 375} (3-4), 1555--1595 (2019).}

\bibitem{StreetsUstinovskiySoliton}{J. Streets and Y. Ustinovskiy, {\sl Classification of generalized Kähler-Ricci solitons on complex surfaces}, Commun. Pure Appl. Math. {\bf 74} (9), 1896--1914 (2020).}

\bibitem{StreetsUstinovskiySolitonII}{J. Streets and Y. Ustinovskiy, {\sl The Gibbons-Hawking ansatz in generalized Kähler geometry}, Commun. Math. Phys. {\bf 391}, 707--778 (2022).}
\end{thebibliography}


\end{document}